 \documentclass[final,1p,times]{elsarticle}
\usepackage{amssymb}
\usepackage{amsthm}

\usepackage{graphicx}
\usepackage{subcaption}
\usepackage{caption}
\usepackage{float}
\usepackage{mathtools}
\usepackage{amssymb}
\usepackage{todonotes}
\usepackage{verbatim}
\usepackage{bbm}
\usepackage{enumitem}
\usepackage{multirow}

\usepackage{tikz}
\usetikzlibrary{patterns}
\usepackage{pgfplots}
\usepgfplotslibrary{colorbrewer}
\pgfplotsset{compat=1.8}
\definecolor{clr1}{RGB}{27,158,119}
\definecolor{clr2}{RGB}{217,95,2}
\definecolor{clr3}{RGB}{117,112,179}
\definecolor{clr4}{RGB}{231,41,138}
\definecolor{clr5}{RGB}{102,166,30}
\definecolor{clr6}{RGB}{230,171,2}
\definecolor{clr7}{RGB}{166,118,29}
\pgfplotsset{
    cycle list={clr1,clr2,clr3,clr4,clr5,clr6,clr7},
}

\newtheorem{property}{Property}
\newtheorem{remark}{Remark}
\newtheorem{theorem}{Theorem}
\newtheorem{conjecture}{Conjecture}
\newtheorem{lemma}{Lemma}
\newtheorem{example}{Example}
\newtheorem{definition}{Definition}
\newtheorem{proposition}{Proposition}

\journal{SIGMETRICS 2021}

\begin{document}

\begin{frontmatter}

%% Title, authors and addresses

%% use the tnoteref command within \title for footnotes;
%% use the tnotetext command for theassociated footnote;
%% use the fnref command within \author or \address for footnotes;
%% use the fntext command for theassociated footnote;
%% use the corref command within \author for corresponding author footnotes;
%% use the cortext command for theassociated footnote;
%% use the ead command for the email address,
%% and the form \ead[url] for the home page:
%% \title{Title\tnoteref{label1}}
%% \tnotetext[label1]{}
%% \author{Name\corref{cor1}\fnref{label2}}
%% \ead{email address}
%% \ead[url]{home page}
%% \fntext[label2]{}
%% \cortext[cor1]{}
%% \address{Address\fnref{label3}}
%% \fntext[label3]{}

\title{Achievable Stability in Redundancy Systems}

%% use optional labels to link authors explicitly to addresses:
%% \author[label1,label2]{}
%% \address[label1]{}
%% \address[label2]{}

\author[label1]{Youri Raaijmakers\corref{cor1}}
\ead{y.raaijmakers@tue.nl}
\author[label1]{Sem Borst}

\address[label1]{Department of Mathematics and Computer Science, Eindhoven University of Technology, 5600 MB Eindhoven, The Netherlands}
\cortext[cor1]{Corresponding author}

\begin{abstract}
We consider a system with $N$~parallel servers where incoming jobs are immediately replicated to, say, $d$~servers.  Each of the $N$ servers has its own queue and follows a FCFS discipline.  As soon as the first job replica is completed, the remaining replicas are abandoned.  We investigate the achievable stability region for a quite general workload model with different job types and heterogeneous servers, reflecting job-server affinity relations which may arise from data locality issues and soft compatibility constraints.
Under the assumption that job types are known beforehand we show for New-Better-than-Used (NBU) distributed speed variations that no replication $(d=1)$ gives a strictly larger stability region than replication $(d>1)$. Strikingly, this does not depend on the underlying distribution of the intrinsic job sizes, but observing the job types is essential for this statement to hold. In case of non-observable job types we show that for New-Worse-than-Used (NWU) distributed speed variations full replication ($d=N$) gives a larger stability region than no replication $(d=1)$.
\end{abstract}

\begin{keyword}
%% keywords here, in the form: keyword \sep keyword
Parallel-server system \sep redundancy \sep stability
%% PACS codes here, in the form: \PACS code \sep code

%% MSC codes here, in the form: \MSC code \sep code
%% or \MSC[2008] code \sep code (2000 is the default)

\end{keyword}

\end{frontmatter}

%% \linenumbers

%% main text

\section{Introduction}
\label{sec: introduction}
Redundancy scheduling has attracted strong interest as a mechanism to improve the delay performance in parallel-server systems. In redundancy scheduling an incoming job is replicated and dispatched to $d$ different servers and as soon as the first of the $d$ replicas finishes service the remaining replicas are abandoned ('cancel-on-completion' c.o.c.). Adding replicas increases the chance for one of the replicas to find a short queue, thus reducing the latency. On the other hand, adding replicas could cause instability since the same job may be in service at multiple servers, potentially wasting service capacity. Establishing the stability condition is not straightforward since the various replicas may have started service at different times. 
Among the numerous studies on redundancy scheduling, stability results have remained scarce so far.
  
\textit{Gardner et al.} \cite{GHBSWVZ-PDR} introduce the redundancy-$d$ system and obtain analytical expressions for the expected number of jobs in the system in the scenario with uniform selection of the servers, exponential job sizes, i.i.d.\ replicas and homogeneous servers, i.e., the server speeds of all servers are equal. From the expressions it is concluded that in this scenario more redundancy is always better for the expected latency. In terms of stability their main result is that the stability condition of the c.o.c.\ version of redundancy scheduling is $\lambda / N \mu < 1$, where $\lambda$ denotes the arrival rate and the job sizes are exponentially distributed with parameter $\mu$. Note that this stability condition is independent of the number of replicas. 

In \cite{RBB-RSSB} it is shown that in the same scenario with scaled Bernoulli job sizes the stability condition is asymptotically given by $\lambda / K^{d-1}< 1$ as $K \rightarrow \infty$. Here the job size is either $0$ or $K$ with probability $1-1/K$ and $1/K$, respectively. Observe that this stability condition is asymptotically independent of the number of servers. 

The contrasting results in~\cite{GHBSWVZ-PDR} and~\cite{RBB-RSSB} indicate that the stability condition is highly sensitive to the job size distribution, and for general job size distributions the stability condition remains unknown. 
For a discrete-time system with Bernoulli arrivals a lower bound is proved in \cite{M-LBSR}. While the bound is not always tight, it provides a first result for the necessary stability condition that depends on the number of servers, the number of replicas and the joint distribution of the job sizes. 

\textit{Anton et al.} \cite{AAJV-OSR} investigate the stability condition in the scenario of homogeneous servers and exponential job sizes for different service disciplines at each individual server, such as processor sharing, FCFS and random order of service. For FCFS and identical replicas they prove an implicit stability condition. Namely, the system is stable if $\lambda/\bar{l} \mu < 1$ and unstable if $\lambda/\bar{l} \mu > 1$, where $\bar{l}$ is the long-run average number of jobs served in the saturated system, i.e., the system with an infinite backlog of jobs. Finding a closed-form expression for $\bar{l}$ remains an open problem. They also explore the stability condition for heterogeneous server speeds by simulation, showing that heterogeneity in server speed has a profound impact on the stability condition.
 
\textit{Gardner et al.} \cite{GHBSW-DSSJS} study the same scenario in the S\&X model, where the server speeds (slowdown factors) at the various servers are independent and identically distributed. No analytical expression is obtained for either the expected latency or the stability condition. However, simulation shows that for more variable job size distributions, more redundancy at first decreases the expected latency, but then hurts badly. The system can even become unstable if the number of replicas $d$ is too high. A dispatching policy 'Redundancy-to-Idle-Queue' (RIQ) that only replicates the job to idle servers is introduced to overcome this problem. Highly accurate approximations for both the expected latency and the transform of the latency are derived. It is proved that, in contrast to redundancy-$d$ scheduling, the RIQ policy cannot become unstable as the number of replicas increases. Stability aspects of redundancy scheduling in a many-server regime are discussed in \cite{HBH-PAWDLB,HH-PR}. For a recent summary of exact stability condition results we refer to \cite[Table~1]{RBB-SRPS}.

Further work has focused on comparing the stability conditions and showing that either no replication ($d=1$) or full replication $(d=N)$ is optimal in the scenario of i.i.d.\ replicas and homogeneous server speeds.  
In \cite{KR-RAGC} it is proved that full replication stochastically maximizes the number of jobs completed jointly across time for NWU job size distributions. No replication is shown to be optimal for two servers and NBU job size distributions, see Definition~\ref{def: NBU NWU distribution} below for the definition of NBU and NWU distributions. In \cite{KRW-JRMS} these results are generalized and it is proved, by a combinatorial argument, that no replication and full replication give the largest stability region for NBU and NWU job sizes, respectively. 
In \cite{J-ERT} these results are extended to log-concave and respectively, log-convex complementary cumulative distribution functions. Note that log-concavity and log-convexity imply NBU and NWU, respectively, but the converse is not true. 

In \cite{WJW-ESR} the single fork-join policy is analyzed. This policy launches $n$ tasks and waits until $(1-p)n$ tasks are finished. For the remaining $p n$ straggling tasks there are two options: either replicate and keep the original task or replicate and kill the original task. Under the assumption that there is no queueing of the tasks it is proved that for NBU distributions keeping the original task gives lower latency while for NWU distributions killing the original task gives lower latency. The effect of replication in the fork-join model is also analyzed in \cite{PC-CER}. Different strategies, such as no replication, full replication or partial replication, are shown to perform better depending on the job size distribution.
In \cite{SKS-DOS} a scheduling policy, called fewest unassigned tasks first with low-priority replication, is proposed in case of an NBU distribution, while the earliest due date first with replication policy is proposed for an NWU distribution. 

In this paper we investigate the achievable stability region for c.o.c.\ redundancy systems in a quite general workload model, as considered earlier in \cite{RBB-DPP}, with multiple job types and servers that follow a FCFS discipline. Replicas may be assigned to the servers according to static type-dependent probabilities, instead of uniformly at random. Additionally, we deal with the complex dynamics arising from potentially different start times as a result of queueing which may occur when servers are not partitioned in disjoint pools of $d$ servers. Specifically, we allow for generally distributed job sizes and the server speeds (slowdown factors) for a given job type are allowed to be inter-dependent and non-identically distributed, reflecting job-server affinity relations which may arise from data locality issues and soft compatibility constraints that are increasingly prevalent in data center environments.
This workload model also subsumes the S\&X model introduced in \cite{GHBSW-DSSJS}.

The general nature of the workload model reveals that the optimal degree of replication is not determined by the distribution of the intrinsic job sizes, but rather by the random variation in service speeds (or slow down factors) for a given job across the various servers. Also, our set-up with different job types and heterogeneous servers separates purely random variation in speeds across servers from systematic differences induced by job-server affinity relations. In particular, our results are the first to demonstrate that when job types are not explicitly observable, this uncertainty plays a similar role as purely random variation in speeds, and creates a potentially strong incentive for replication, even when the speeds for a job of a given type show little or no variation at all. Conversely, if there is little or no random variation in speeds, and the variability primarily arises from fundamental heterogeneity in job characteristics that can be observed beforehand, then replication provides no gains from a stability perspective.

The remainder of the paper is organized as follows. In Section \ref{sec: model description} we present a detailed model description and some preliminary results. In Sections~\ref{sec: stability NBU distribution} and~\ref{sec: stability NWU distribution} we state and prove the main theorems for NBU and NWU distributed speed variations, respectively. Section \ref{sec: conclusion} contains conclusions and some suggestions for further research. 

\section{Model description and preliminary results}
\label{sec: model description}

Consider a system with $N$ parallel servers where jobs arrive as a Poisson process of rate $\lambda$. Each of the $N$ servers has its own queue and follows a FCFS discipline. When a job arrives, multiple replicas may be assigned to one or more servers according to static type-dependent probabilities. 
A special case of such a static probabilistic assignment is the celebrated \textit{power-of-$d$} policy, where replicas are assigned to $d$ servers selected uniformly at random (without replacement), which is the prevalent case considered in the literature. 

In case multiple replicas are assigned, the service speeds $R_{1},\dots,R_{N}$ for that job on the various servers may differ. We allow the service speeds $R_1, \dots, R_{N}$ of a generic job to be governed by some joint distribution $F(r_1, \dots, r_{N})$, reflecting possible server heterogeneity and job-server compatibility relations. For convenience, we consider the case where the joint distribution $F(r_1, \dots, r_{N})$ is discrete, and has mass in a finite number of say~$M$ points $(r_{1j},...,r_{Nj})$ with corresponding probabilities $p_{j}$, $j= 1, \dots,M$. This system may equivalently be thought of as having $M$ job types, where $r_{ij}$ is the service speed of type-$j$ jobs at the $i$-th server. 
For notational convenience let $\mathcal{N} = \{1,2,\dots,N\}$ denote the set of servers and $\mathcal{M}=\{1,2,\dots,M\}$ denote the set of job types. 
 
The intrinsic size of a type-$j$ job is denoted by a generic random variable $X_{j}$. Moreover, letting $Y_{ij}$ denote the random speed variation, we assume that $Y_{1j},\dots,Y_{Nj}$ are i.i.d.\ copies of some generic random variable $S_{j}$. These latter variables can be thought of as job sizes in the standard independent runtime model (taking $X_{j}=1$ and $M=1$ job types) or slowdown factors in the S\&X model \cite{GHBSW-DSSJS} (taking $R_{i}=1$ and $M=1$ job types). For a particular job on server $i$, $i=1,\dots,N$, with intrinsic size $x_{j}$, $(x_{j} Y_{ij})/R_{i}$ represents the processing time. We distinguish two cases: i) no random speed variation for all job types, i.e., $S_{j} \equiv c_{j}$ with $c_{j} \in \mathbb{R}_{+}$ for $j=1,\dots,M$, so-called identical replicas, ii) random speed variation for all job types and servers, so-called i.i.d.\ replicas. 

In the remainder of the paper we distinguish between two scenarios referred to as \textit{Known job types} and \textit{Unknown job types}. In both scenarios the design of the assignment policy may involve knowledge of the type probabilities $p_j$ and service speeds $r_{ij}$. In the Known job types case, the dispatcher can additionally observe the type identity of each job, and thus knows its service speed at each of the servers.  In contrast, in the Unknown job types case, the dispatcher cannot identify jobs by type, and thus has no advance knowledge of service speeds of individual jobs.

\subsection{Preliminaries}
\label{sec: preliminaries}
Let $\tilde{p}_{ij}$ denote the proportion of type-$j$ jobs that are assigned to server $i$. 
For given $\tilde{p}_{ij}$, the stability condition for $d=1$ and known job types, see also \cite{HL-RPPS,S-ORSS}, is given by $\lambda \sum_{j=1}^{M} \tilde{p}_{ij} p_{j}  \frac{\mathbb{E}[X_{j}]\mathbb{E}[S_{j}]}{r_{ij}} < 1$ for all $i=1,\dots,N$. 
Thus, the achievable stability region is
\begin{align}
\Lambda_{K} &= \Bigg\{ \lambda \geq 0 | \exists \tilde{p}_{ij} \geq 0 : \lambda \sum_{j=1}^{M} \tilde{p}_{ij} p_{j}  \frac{\mathbb{E}[X_{j}]\mathbb{E}[S_{j}]}{r_{ij}} < 1 \text{ for all } i \in \mathcal{N}, \sum_{i=1}^{N}\tilde{p}_{ij} = 1 \text{ for all } j \in \mathcal{M} \Bigg\},
\label{eq: stability condition d=1 job types}
\end{align}
where the subscript $K$ refers to the case of \textit{known} job types. Note that the stability region given by Equation~\eqref{eq: stability condition d=1 job types} only depends on the distribution of $S_{j}$ through its mean $\mathbb{E}[S_{j}]$ since there is no replication.

Now we proceed with the case of \textit{unknown} job types. Let $\tilde{p}_{i}$ denote the proportion of jobs assigned to server $i$, which must be common to all job types when these cannot be distinguished.
For given $\tilde{p}_{i}$, the stability condition for $d=1$ is then given by $\sum_{j=1}^{M}  \lambda \tilde{p}_{i} p_{j} \frac{\mathbb{E}[X_{j}]\mathbb{E}[S_{j}]}{r_{ij}} < 1$ for all $i=1,\dots,N$.
Thus, the achievable stability region is
\begin{align}
\Lambda_{U} &= \Bigg\{ \lambda \geq 0 | \exists \tilde{p}_{i} \geq 0 :\sum_{j=1}^{M}  \lambda \tilde{p}_{i} p_{j} \frac{\mathbb{E}[X_{j}]\mathbb{E}[S_{j}]}{r_{ij}} < 1 \text{ for all } i \in \mathcal{N}, \tilde{p}_{1} + \dots + \tilde{p}_{N}=1 \Bigg\},
\label{eq: stability condition d=1 no job types}
\end{align}
where the subscript $U$ refers to the case of \textit{unknown} job types.

The stability region for $d=N$ is also known since the system then behaves as an $M/G/1$ system, see for example~\cite{AAJV-OSR},
\begin{align}
\Lambda = [0,\lambda^{*}),
\label{eq: stability condition d=N no job types}
\end{align}
with $\lambda^{*} = \left( \sum_{j=1}^{M} p_{j} \mathbb{E}[X_{j}] \mathbb{E}[\min \{\frac{Y_{1j}}{r_{1j}},\dots,\frac{Y_{Nj}}{r_{Nj}} \} \right)^{-1}$.
Note that $\Lambda$ needs no subscript since the stability region is the same in the cases of known and unknown job types.

In the case of generally distributed job sizes, the next example shows that there is a scenario in which the stability region for $d=1$ is strictly larger than for $d=N=2$, both for identical and i.i.d.\ replicas.

\begin{example}
\label{exam: slow fast servers}
Consider the scenario with $N=2$, $M=2$ and server speeds $(r_{11},r_{21})=(1,x)$ and $(r_{12},r_{22})=(x,1)$ with probabilities $p_{1}=p_{2}=0.5$, where $x<1$. In case of $d=1$ the optimal static probabilistic assignment is $\tilde{p}_{11}=\tilde{p}_{22}=1$ and $\tilde{p}_{12}=\tilde{p}_{21}=0$. Thus, the stability conditions are
\begin{align*}
&\lambda \sum_{j=1}^{2} p_{j} \mathbb{E}[X_{j}]\mathbb{E}[S_{j}] < 1, &\text{ for } d=2 \text{ (identical)},\\
&\lambda \sum_{j=1}^{2} p_{j} \mathbb{E}[X_{j}] \mathbb{E}\left[\min\left\{Y_{1j},\tfrac{Y_{2j}}{x}\right\}\right] < 1, &\text{ for } d=2 \text{ (i.i.d.)},\\
&\lambda \cdot \frac{1}{2} \cdot \mathbb{E}[X_{j}]\mathbb{E}[S_{j}] < 1,~~~j=1,2 &\text{ for } d=1,
\end{align*}
where $Y_{1j}$ and $Y_{2j}$ are i.i.d.\ copies of $S_{j}$. Moreover observe that
\begin{align*}
\lim_{x \downarrow 0} \mathbb{E}\left[\min\left\{Y_{1j},\frac{Y_{2j}}{x}\right\}\right] = \mathbb{E}\left[\lim_{x \downarrow 0} \min\left\{Y_{1j},\frac{Y_{2j}}{x}\right\}\right] = \mathbb{E}[S_{j}],
\end{align*}
for every distribution of the speed variation. Thus, if $\mathbb{E}[X_{j}]\mathbb{E}[S_{j}]=1$ for $j=1,2$, then the stability condition for $d=1$ is given by $\lambda<2$ while for $d=2$ it is given by $\lambda<1$.
\end{example}

\begin{definition}
\label{def: NBU NWU distribution}
Consider a non-negative random variable $S$ with support denoted by $\mathcal{R}_{S}$ and cumulative distribution function (cdf) $F_{S}(x)$. Let $\bar{F}_{S}(x) = 1-F_{S}(x)$ denote the complementary cumulative distribution function (ccdf). Then, $S$ is New-Better-than-Used (NBU) if for all $t_{1}, t_{2} \in \mathcal{R}_{S}$,
\begin{align}
\label{eq: NBU definition}
\bar{F}_{S}(t_{1}+t_{2}) \leq \bar{F}_{S}(t_{1}) \bar{F}_{S}(t_{2}).
\end{align}
On the other hand, $S$ is New-Worse-than-Used (NWU) if for all $t_{1}, t_{2} \in \mathcal{R}_{S}$,
\begin{align}
\label{eq: NWU definition}
\bar{F}_{S}(t_{1}+t_{2}) \geq \bar{F}_{S}(t_{1}) \bar{F}_{S}(t_{2}).
\end{align}
Moreover, $S$ is strictly NBU or strictly NWU when Equation \eqref{eq: NBU definition} or \eqref{eq: NWU definition} holds with strict inequality, respectively, for all values $t_{1},t_{2} \in \mathcal{R}_{S}\setminus \{0\}$. 
\end{definition}

In the case of strictly NWU distributed speed variations, the next example shows that there is a scenario in which the stability region for $d=N=2$ is strictly larger than for $d=1$.

\begin{example}
\label{exam: homogeneous servers}
Consider the scenario with $N=2$, $M=1$ and server speeds $(r_{11},r_{21})=(1,1)$ with probability $p_{1}=1$. The stability conditions for $d=1$ and $d=2$ are
\begin{align*}
&\lambda \cdot \frac{1}{2} \cdot \mathbb{E}[X]\mathbb{E}[S] < 1, &\text{ for } d=1,\\
&\lambda \mathbb{E}[X] \mathbb{E}[\min \{ Y_{1}, Y_{2} \} ] < 1, &\text{ for } d=2 \text{ (i.i.d.)},
\end{align*}
where $Y_{1}$ and $Y_{2}$ are i.i.d.\ copies of $S$. Moreover, by definition of strictly NWU, see for example~\cite[Sec.~1.6]{S-CMQST},
\begin{align*}
\mathbb{E}[\min \{ Y_{1}, Y_{2} \} ] < \frac{1}{2} \mathbb{E}[S].
\end{align*}

Thus the stability region for $d=N=2$ is strictly larger than the stability region for $d=1$ in this example.
\end{example}

Observe that $G_{j}(d) = d \mathbb{E}[\min \{Y_{1j},\dots,Y_{dj}\}]$ is increasing and decreasing in $d$ for NBU and NWU distributions, respectively, see for example~\cite[Sec.~1.6]{S-CMQST}. 
Here, $G_{j}(d)$ may be interpreted as the aggregate resource usage for $d$ replicas with equal start times on homogeneous servers under redundancy c.o.c., and has emerged as a key metric for stability conditions in scenarios where the servers are partitioned in disjoint pools of $d$ servers, see for instance \cite{J-ERT}.
We will extend this notion to scenarios with heterogeneous servers and additionally deal with the complex dynamics arising from potentially different start times as a result of queueing which may occur when servers are not partitioned in the above manner.

In the proofs of the main theorems in the next section a property of c.o.c.\ redundancy systems, viz., Property~\ref{prop: oldest job} below, is needed. Note that this property is valid for all scenarios. 
\begin{property}
\label{prop: oldest job}
The oldest job in the system is served at all servers that it has been replicated to.
\end{property}

\section{No replication is best for NBU speed variations}
\label{sec: stability NBU distribution}
In this section we prove that no replication maximizes stability when the speed variations are NBU distributed, see Theorem~\ref{thm: maximum stability d=1 iid NBU}. First however we consider the special case where the speed variation of each job type $j$, $S_{j}$, follows a degenerate distribution, see Theorem~\ref{thm: maximum stability d=1 identical}. The proof of this latter theorem is simpler and gives intuition for the general case with NBU distributed speed variations. Both theorems rely on the next proposition.

\begin{proposition}
\label{propo: prob allocation}
Assume that $r_{ij} > 0$ for all $i = 1, \dots, N$, $j = 1, \dots, M$,
and that the system is stable under a given assignment policy with $d > 1$
for some arrival rate $\lambda_0 > 0$.
Let $\tau_{ij}$ be the long-term fraction of time that server~$i$ spends on type-$j$ jobs under this assignment policy with $d>1$.
Suppose that
\begin{align}
\sum_{i=1}^{N} r_{ij} \tau_{ij} \geq \lambda_{0} p_j \mathbb{E}[X_j] \mathbb{E}[S_j],
\label{ineqpertype}
\end{align}
for all $j =1, \dots, M$, and in addition
\begin{align}
\sum_{j=1}^{M} \sum_{i=1}^{N} r_{ij} \tau_{ij} \geq \lambda_0 (1+\epsilon) \sum_{j=1}^{M} p_j \mathbb{E}[X_j] \mathbb{E}[S_j]
\label{ineqtotal}
\end{align}
for some $\epsilon \geq 0$.
Then the system can be stabilized through a suitable probabilistic assignment policy with $d = 1$ for all $\lambda \leq \lambda_0 (1 + \kappa \epsilon)$,
with $\kappa > 0$ a fixed constant bounded away from zero, independent of~$\lambda_0$.
\end{proposition}
\begin{proof}
The high-level idea of the proof may be outlined as follows.
The inequalities in~\eqref{ineqpertype} imply that the weighted fraction of time that the servers collectively spend on type-$j$ jobs under the policy with $d > 1$ is no less than the offered load of type-$j$ jobs, i.e., what this weighted fraction would be without any replication, for each job type $j = 1, \dots, M$.
This allows us to distribute the type-$j$ jobs without any replication through suitable assignment probabilities $\tilde{p}_{ij}$ in terms of the $r_{ij}$ and $\tau_{ij}$ to sustain the same arrival rate~$\lambda_0$ without increasing the load of any of the servers, thus ensuring stability.
Hence, the statement of the proposition follows when $\epsilon = 0$.
When $\epsilon > 0$, the inequality~\eqref{ineqtotal} implies that the total weighted amount of time that the servers are collectively occupied under the policy with $d > 1$ is strictly larger than the total offered load.
This margin reflects that there is some slack capacity that can be freed up when refraining from replication, and thus be exploited to accommodate a higher arrival rate than~$\lambda_0$.
While there are several options for dividing the slack capacity, we will simply use assignment probabilities that account for the amount of slack at each server and its speeds for the various job types, but do not depend on the job type.
Once again, this will not increase the load of any of the servers, but allow us to support a strictly higher arrival rate.

In order to develop the proof in greater detail, observe that the stability under the given assignment policy with $d > 1$ implies that the long-term fraction of time that each server is busy must be strictly less than unity, i.e., $\sum_{j=1}^{M} \tau_{ij} < 1$ for all $i = 1, \dots, N$.
(For transparency, we tacitly assume here and in the statement of the proposition that these long-term fractions exist, and thus implicitly rule out possibly eccentric (e.g. non-stationary) assignment policies.
The proof arguments below could however readily be extended to cover such policies as well, if we stipulate stability to mean that the limsup values of $\sum_{j=1}^{M} \tau_{ij}$ must be strictly less than unity for all $i = 1, \dots, N$.)

Now consider the system with $d=1$ and assignment probabilities
\begin{align*}
\tilde{p}_{ij} = \frac{r_{ij} \tau_{ij}}{\sum_{k=1}^{N} r_{kj} \tau_{kj}}.
\end{align*}

Then each server behaves as a multi-class $M/G/1$ queue, and for an overall arrival rate $\lambda \leq \lambda_0$ the load on server~$i$ is
\begin{align*}
\lambda \sum_{j=1}^{M} p_j \tilde{p}_{ij} \frac{\mathbb{E}[X_j] \mathbb{E}[S_{j}]}{r_{ij}}
= \lambda \sum_{j=1}^{M} p_j \tau_{ij} \frac{\mathbb{E}[X_j] \mathbb{E}[S_{j}]}{\sum_{k=1}^{N} r_{kj} \tau_{kj}}
\leq \sum_{j=1}^{M} \tau_{ij} < 1, ~~~ \forall i = 1, \dots, N,
\end{align*}
where the last-but-one inequality follows from~\eqref{ineqpertype} and the fact that $\lambda \leq \lambda_0$, implying that the system is stable.
This completes the proof in case $\epsilon = 0$.

In order to prove the statement in case $\epsilon > 0$, let
\begin{align*}
\sigma_j = \frac{\lambda_0 p_j \mathbb{E}[X_j] \mathbb{E}[S_j]}{\sum_{i=1}^{N} r_{ij} \tau_{ij}} \leq 1,
\end{align*}
representing the offered load of type-$j$ jobs as fraction of the weighted amount of time spent on these jobs by the servers collectively under the given assignment policy with $d>1$, and define
\begin{align*}
\hat\tau_{ij} &= \sigma_j \tau_{ij} \leq \tau_{ij}, \\
\Delta\tau_{ij} &= \tau_{ij} - \hat\tau_{ij} = (1 - \sigma_j) \tau_{ij},
\end{align*}
and
\begin{align*}
\Delta\tau_i = \sum_{j=1}^{M} \tau_{ij} - \sum_{j=1}^{M} \hat\tau_{ij} = \sum_{j=1}^{M} \Delta\tau_{ij}.
\end{align*}
The value of~$\hat\tau_{ij}$ may be interpreted as the fraction of time that server~$i$ would need to spend on type-$j$ jobs if the efforts of all servers for type-$j$ jobs are reduced proportionally to match the total offered load.
With that interpretation in mind, $\Delta\tau_{ij}$ and $\Delta\tau_i$ may be thought of as measures for the slack capacity.

Further introduce
\[
\Delta_j = \sum_{i=1}^{N} r_{ij} \tau_{ij} - \lambda_0 p_j \mathbb{E}[X_j] \mathbb{E}[S_j] = \sum_{i=1}^{N} r_{ij} \tau_{ij} - \sum_{i=1}^{N} r_{ij} \hat\tau_{ij} = \sum_{i=1}^{N} r_{ij} \Delta\tau_{ij}
\]
representing the slack between the weighted fraction of time that the servers collectively spend on type-$j$ jobs under the policy with $d > 1$ and the offered load of type-$j$ jobs,
\[
r_i = \frac{\sum_{j=1}^{M} p_j \mathbb{E}[X_j] \mathbb{E}[S_j]}{\sum_{j=1}^{M} p_j \mathbb{E}[X_j] \mathbb{E}[S_j] / r_{ij}}
\]
representing the time-average speed of server~$i$ when handling jobs of the various types in the nominal proportions, and
\begin{align*}
\Delta\lambda = \frac{\sum_{k=1}^{N} r_k \Delta\tau_k}{\sum_{j=1}^{M} p_j \mathbb{E}[X_j] \mathbb{E}[S_j]}.
\end{align*}

Now observe that on the one hand
\begin{align*}
\sum_{j=1}^{M} \Delta_j = \sum_{j=1}^{M} \sum_{i=1}^{N} r_{ij} \tau_{ij} - \lambda_0 \sum_{j=1}^{M} p_j \mathbb{E}[X_j] \mathbb{E}[S_j] \geq
\lambda_0 \epsilon \sum_{j=1}^{M} p_j \mathbb{E}[X_j] \mathbb{E}[S_j],
\end{align*}
while on the other hand
\begin{align*}
\sum_{j=1}^{M} \Delta_j = \sum_{j=1}^{M} \sum_{i=1}^{N} r_{ij} \Delta\tau_{ij} \leq \sum_{i=1}^{N} \Delta\tau_i \max_{j \in \mathcal{M}} r_{ij},
\end{align*}
and hence
\[
\Delta\lambda \geq \frac{\sum_{k=1}^{N} r_k \Delta\tau_k}{\sum_{i=1}^{N} \Delta\tau_i \max_{j \in \mathcal{M}} r_{ij}} \epsilon \lambda_{0}.
\]
Noting that $r_i > 0$ by virtue of the assumption that $r_{ij} > 0$ for all $i = 1, \dots, N$ and $j = 1, \dots, M$, we obtain that
\begin{align*}
\Delta\lambda \geq \kappa \epsilon \lambda_0,
\end{align*}
with $\kappa = \frac{\min_{i \in \mathcal{N}} r_i}{\max_{i \in \mathcal{N}, j \in \mathcal{M}} r_{ij}} > 0$.

Now consider the system with $d=1$ and total arrival rate $\lambda_0 + \Delta\lambda$, and suppose that a fraction $\lambda_0 / (\lambda_0 + \Delta\lambda)$ of the jobs are assigned according to the probabilities $\tilde{p}_{ij}$, while the remaining fraction $\Delta\lambda / (\lambda_0 + \Delta \lambda)$ of the jobs are assigned to server~$i$ with probability
\begin{align*}
\hat{p}_i = \frac{r_i \Delta\tau_i}{\sum_{k=1}^{N} r_k \Delta\tau_k}.
\end{align*}
Then each server behaves as a multi-class $M/G/1$ queue, and for an overall arrival rate $\lambda \leq \lambda_0 + \Delta\lambda$ the load on server~$i$ is
\begin{align*}
& \lambda \sum_{j=1}^{M} p_j \left(\frac{\lambda_0}{\lambda_0 + \Delta\lambda} \tilde{p}_{ij} + \frac{\Delta\lambda}{\lambda_0 + \Delta\lambda} \hat{p}_i\right) \frac{\mathbb{E}[X_j] \mathbb{E}[S_j]}{r_{ij}} \\
&= \frac{\lambda}{\lambda_0 + \Delta\lambda} \sum_{j=1}^{M} p_j \left(\lambda_0
\frac{r_{ij} \tau_{ij}}{\sum_{k=1}^{N} r_{kj} \tau_{kj}} + \Delta\lambda \frac{r_i \Delta\tau_i}{\sum_{k=1}^{N} r_k \Delta\tau_k}\right) \frac{\mathbb{E}[X_j] \mathbb{E}[S_j]}{r_{ij}} \\
&\leq \sum_{j=1}^{M} \tau_{ij} \frac{\lambda_0 p_j \mathbb{E}[X_j] \mathbb{E}[S_j]}{\sum_{k=1}^{N} r_{kj} \tau_{kj}} + \frac{\sum_{k=1}^{N} r_k \Delta\tau_k}{\sum_{j=1}^{M} p_j \mathbb{E}[X_j] \mathbb{E}[S_j]} \frac{r_i \Delta\tau_i}{\sum_{k=1}^{N} r_k \Delta\tau_k} \sum_{j=1}^{M} \frac{p_j \mathbb{E}[X_j] \mathbb{E}[S_j]}{r_{ij}} \\
&= \sum_{j=1}^{M} \tau_{ij} \sigma_j + r_i \Delta\tau_i \frac{\sum_{j = 1}^{M} p_j \mathbb{E}[X_j] \mathbb{E}[S_j] / r_{ij}}{\sum_{j=1}^{M} p_j \mathbb{E}[X_j] \mathbb{E}[S_j]} \\
&= \sum_{j=1}^{M} \hat\tau_{ij} + \Delta\tau_i = \sum_{j=1}^{M} \tau_{ij} < 1, ~~~ \forall i = 1, \dots, N,
\end{align*}
where the inequality in the third line follows from the fact that $\lambda \leq \lambda_0 + \Delta\lambda$.

This yields the statement of the proposition for any $\epsilon \geq 0$.
\end{proof}

\begin{remark}
We now present an example illustrating the role of the assumption that $r_{ij} > 0$ for all $i = 1, \dots, N$ and $j = 1, \dots, M$.  Consider a system with $N = 3$ servers, $M = 2$ job types, and service speeds $(r_{11}, r_{21}, r_{31}) = (1, 0, 0)$ and $(r_{12}, r_{22}, r_{32}) = (1, 1, 1)$.  Assume that $p_j = 1/2$, $E[S_j] = 1$, $E[X_j] = 1$, $j = 1, 2$, $d = 2$, and that all type-$1$ jobs are assigned to servers~$1$ and~$3$ while all type-$2$ jobs are assigned to servers~$2$ and~$3$.  We claim that the system is stable for any $\lambda < 2$.  In order to see that, observe that the number of type-$1$ jobs and the number of type-$2$ jobs are each individually bounded from above by the number of the jobs in an $M/G/1$ queue with load $\lambda / 2$.  Furthermore, type-$1$ jobs will never complete on server~$3$,  while in case of identical service times, type-$2$ jobs will never complete on server~$3$ before completing on server~$2$.  In other words, all effort of server~$3$ goes wasted.  Nevertheless, a system with $d = 1$ cannot be stabilized for any $\lambda \geq 2$, since type-1 jobs can only be successfully processed by server~$1$.  The wasted effort of server~$3$ could however be avoided in a system with $d = 1$ to sustain an arrival rate of type-$2$ jobs that is twice as large. 

While the assumption that $r_{ij} > 0$ for all $i = 1, \dots, N$ may in general not be strictly necessary, this example demonstrates that it cannot easily be relaxed without creating a need for a tedious case-by-case analysis to determine whether the system with $d = 1$ can be stabilized for a higher overall arrival rate, can only accommodate a larger arrival rate for some of the job types, or cannot support a higher arrival rate for any job type at all.
\end{remark}

\begin{theorem}
\label{thm: maximum stability d=1 identical}
In the case of known job types, the stability region for $d=1$ is strictly larger than the stability region for $d>1$ under the c.o.c.\ redundancy policy with identical replicas and static probabilistic assignment (which may depend on the job type) of the $d$ replicas.  
\end{theorem}
\begin{proof}
Let $\tau_{ij}^{(1)}$ be the fraction of time that server $i$ spends on type-$j$ jobs that it will finish and $\tau_{ij}^{(2)}$ be the fraction of time that server $i$ spends on type-$j$ jobs that it will not finish, with $\tau_{ij}^{(1)}+\tau_{ij}^{(2)}=\tau_{ij}$ under a given assignment policy with $d>1$ for arrival rate $\lambda$.

For the effective component we have
\begin{align}
\label{eq: effictive component condition}
\sum_{i=1}^{N} r_{ij} \tau_{ij}^{(1)} = \lambda p_{j}\mathbb{E}[X_{j}] \mathbb{E}[S_{j}],
\end{align}
since
\begin{align*}
\sum_{i=1}^{N}  r_{ij} \mathbb{E}[T_{ij}^{(1)}] = \mathbb{E}[X_{j}] \mathbb{E}[S_{j}],
\end{align*}
where $T_{ij}^{(1)}$ is the amount of time that server $i$ spends on a type-$j$ job that it will finish, with $\tau_{ij}^{(1)} = \lambda p_{j} \mathbb{E}[T_{ij}^{(1)}]$. This holds because for identical replicas there are no server-dependent slow downs and whether or not a server will finish a particular job is not influenced by the random speed variations.

For the wastage component we have by Property~\ref{prop: oldest job} that
\begin{align}
\sum_{i=1}^{N} r_{ij} \tau_{ij}^{(2)} &\geq \sum_{i=1}^{N} \min_{k \in \mathcal{N}}r_{kj} \tau_{ij}^{(2)} = \min_{k \in \mathcal{N}}r_{kj}\sum_{i=1}^{N} \tau_{ij}^{(2)} \geq \min_{k \in \mathcal{N}}r_{kj} (d-1) \bar{\pi}_{0j},
\label{eq: thm1 wastage component}
\end{align}
where $\bar{\pi}_{0j}$ is the fraction of time that the system is non-empty in the limit as time goes to infinity and the oldest job is of type $j$. Letting $\bar{\pi}_{0}$ be the fraction of time that the system is non-empty with $\sum_{j=1}^{M} \bar{\pi}_{0j} = \bar{\pi}_{0}$, it follows that
\begin{align*}
\sum_{j=1}^{M} \sum_{i=1}^{N} r_{ij} \tau_{ij}^{(2)} \geq \min_{k \in \mathcal{N}, l \in \mathcal{M}}r_{kl} (d-1) \bar{\pi}_{0}.
\end{align*}
We can bound the fraction of time that the system is non-empty as
\begin{align*}
\bar{\pi}_{0} \geq \sum_{j=1}^{M} \frac{\lambda p_{j} \mathbb{E}[X_{j}] \mathbb{E}[S_{j}]}{\sum_{i=1}^{N} r_{ij}} \geq \frac{\lambda \sum_{j=1}^{M} p_{j} \mathbb{E}[X_{j}]\mathbb{E}[S_{j}]}{\max_{j \in \mathcal{M}} \sum_{i=1}^{N} r_{ij}} > 0.
\end{align*} 
Substituting this in Equation~\eqref{eq: thm1 wastage component} gives
\begin{align*}
%\label{eq: thm1 M fraction time}
\sum_{j=1}^{M} \sum_{i=1}^{N} r_{ij} \tau_{ij} \geq \lambda \left(1+  (d-1)\tfrac{\min_{k \in \mathcal{N}, l \in \mathcal{M}}r_{kl}}{\max_{j \in \mathcal{M}} \sum_{i=1}^{N} r_{ij}}\right) \sum_{j=1}^{M} p_{j} \mathbb{E}[X_{j}] \mathbb{E}[S_{j}],
\end{align*} 
so that Equation~\eqref{ineqtotal} holds with $\epsilon = (d-1) \frac{\min_{k \in \mathcal{N}, l \in \mathcal{M}}r_{kl}}{\max_{j \in \mathcal{M}} \sum_{i=1}^{N} r_{ij}}$ which is bounded away from zero. Noting that~\eqref{eq: effictive component condition} with in addition $\tau_{ij}^{(1)}+\tau_{ij}^{(2)}=\tau_{ij}$ gives~\eqref{ineqpertype}, the proof then follows from Proposition~\ref{propo: prob allocation}.
\end{proof}

\begin{remark}
In Theorem~\ref{thm: maximum stability d=1 identical} we obtained a lower bound for the wastage component that is strictly increasing in $d$, see Equation~\eqref{eq: thm1 wastage component}. We can also find an upper bound for the wastage component
\begin{align}
\sum_{j=1}^{M} \sum_{i=1}^{N} r_{ij} \tau^{(2)}_{ij} &\leq \max_{k \in \mathcal{N}, l \in \mathcal{M}} r_{kl} \left(N - \left\lceil \frac{N}{d} \right\rceil \right) \bar{\pi}^{*}_{0} \leq \max_{k \in \mathcal{N}, l \in \mathcal{M}} r_{kl} \left(N - \left\lceil \frac{N}{d} \right\rceil \right) \bar{\pi}_{0},
\label{eq: rem1 wastage component}
\end{align}
where $\bar{\pi}^{*}_{0}$ is the fraction of time that all servers are busy and thus $\bar{\pi}^{*}_{0} \leq \bar{\pi}_{0}$. 
Note that in the special case of homogeneous server speeds and $d=1$, $d=N-1$ and $d=N$ the lower- and upper bound for the wastage component coincide. 
It is therefore natural to conjecture that Theorem~\ref{thm: maximum stability d=1 identical} extends to the statement that the stability region is strictly decreasing in $d$.
\end{remark}

\begin{comment}
\begin{corollary}
The stability condition for $d=N-1$ odd number of identical replicas, general job sizes and homogeneous servers is given by $\lambda \mathbb{E}[X] < 2$.
\end{corollary}
\begin{proof}
From the coinciding lower- and upper bound for the wastage component in Equations~\eqref{eq: thm1 wastage component} and~\eqref{eq: rem1 wastage component}, respectively, we can conclude that the wastage component is
\begin{align*}
\sum_{j=1}^{M} \sum_{i=1}^{N} r_{ij} \tau^{(2)}_{ij} = (N-2) \bar{\pi}_{0}.
\end{align*}
First observe that the wastage component for the system with $d=N$ is equal to 
\begin{align*}
\sum_{i=1}^{N} r_{ij} \tau_{ij}^{(2)} = (N-1) \pi_{0},
\end{align*}
since all replicas start at exactly the same time. Now, consider the system with $N$ servers that are divided into two server pools of size $N/2$, where we allocate $d=N/2$ replicas to one of the two server pools uniformly at random. Within these server pools all the replicas start at the same time and therefore this sub-system behaves as an $M/G/1$ queue with arrival rate $\lambda/2$ and job size $X$. The stability condition of the sub-system is therefore given by $\lambda \mathbb{E}[X] < 2$. The wastage component for this sub-system is $(N/2-1) \pi_{0}$, which means that the wastage component for the whole system is $(N-2) \pi_{0}$. From the coinciding wastage components of this system and the system with $d=N-1$ we conclude that also the stability conditions must coincide. 
\end{proof}
\end{comment}

We proceed with the general case of speed variations that are NBU distributed. 

\begin{theorem}
\label{thm: maximum stability d=1 iid NBU}
In the case of known job types, the stability region for $d=1$ is larger than or equal to (respectively, strictly larger than) the stability region for $d>1$ with NBU (respectively, strictly NBU) distributed speed variations and static probabilistic assignment (which may depend on the job type) of the $d$ replicas. 
\end{theorem}

\begin{proof}
Let $T_{ij}$ be the amount of time that server $i$ spends on an arbitrary type-$j$ job and let $\tau_{ij}$ be the fraction of time that server $i$ spends on type-$j$ jobs under a given assignment policy with $d>1$ for arrival rate $\lambda$ as introduced before.

Let $T_{\mathrm{awt},j}$ with distribution function $F_{T_{\mathrm{awt},j}}(t)$ (respectively, $T_{\mathrm{awt},j}^{I}$ with distribution function $F_{T_{\mathrm{awt},j}^{I}}(t)$) denote the aggregate weighted amount of time, weighted by the server speeds $r_{ij}$, invested in the service of an arbitrary type-$j$ job divided by the intrinsic job size (respectively, under the assignment $I$, where $I$ denotes an arbitrary set of $d$ servers). We have that $T_{\mathrm{awt},j}^{I}$ is equal in distribution to $\sum_{i \in I} \frac{r_{ij} T_{ij}}{X_{j}}$ (see Figure~\ref{fig: T overlap} for a schematic illustration), when server $I_{i}$ is available after the weighted amount of time $b_{i}$ of servers $I_{1},\dots,I_{i-1}$, for $i=2,\dots,d$ and $b_{1}=0$. Thus, a replica of the job is first served on server $I_{1}$ and after time $b_{2}$ server $I_{2}$ becomes available to serve another replica of this job, then after time $b_{3}$ the third server $I_{3}$ becomes available to serve yet another replica of this job, etc. Note that server $I_{i}$ may not necessarily serve this job, i.e., the job may already be completed before the server is available. The ccdf is
\begin{align*}
\bar{F}_{T_{\mathrm{awt},j}^{I}}(t) =
\begin{cases}
\mathbb{P}\left(Y_{I_{1}j} > \frac{r_{I_{1}j} t}{r_{I_{1}j}}\right) & \text{ for } 0 < t < b_{2}, \\
\mathbb{P}\left(Y_{I_{1}j} > b_{2} +\frac{r_{I_{1}j}(t-b_{2})}{r_{I_{1}j}+r_{I_{2}j}}\right) \cdot \mathbb{P}\left( Y_{I_{2}j} >  \frac{r_{I_{2}j}(t-b_{2})}{r_{I_{1}j}+r_{I_{2}j}}\right)  & \text{ for }b_{2} < t < \sum_{l=1}^{3} (b_{3}-b_{l}), \\
\vdots \\
\mathbb{P}\left( Y_{I_{1}j} > b_{2}+\frac{r_{I_{1}j}2(b_{3}-b_{2})}{r_{I_{1}j}+r_{I_{2}j}}+ \cdots + \frac{r_{I_{1}j}(t-\sum_{l=1}^{d}(b_{d}-b_{l}))}{\sum_{i \in I} r_{ij}}\right) \cdots \\ 
\quad \cdot \mathbb{P}\left( Y_{I_{d}j} >  \frac{r_{I_{d}j}(t-\sum_{l=1}^{d}(b_{d}-b_{l}))}{\sum_{i \in I} r_{ij}}\right) & \text{ for } \sum_{l=1}^{d}(b_{d}-b_{l}) < t.
\end{cases}
\end{align*}
Hence
\begin{align*}
\bar{F}_{T_{\mathrm{awt},j}^{I}}(t) =
\begin{cases}
\bar{F}_{Y_{I_{1}j}}\left(t \right) & \text{ for } 0 < t < b_{2}, \\
\bar{F}_{Y_{I_{1}j}}\left( b_{2}+\frac{r_{I_{1}j}(t-b_{2})}{r_{I_{1}j}+r_{I_{2}j}}\right) \cdot \bar{F}_{Y_{I_{2}j}}\left(\frac{r_{I_{2}j}(t-b_{2})}{r_{I_{1}j}+r_{I_{2}j}}\right)  & \text{ for }b_{2} < t < \sum_{l=1}^{3} (b_{3}-b_{l}), \\
\vdots \\
\bar{F}_{Y_{I_{1}j}}\left(b_{2}+\frac{r_{I_{1}j}2(b_{3}-b_{2})}{r_{I_{1}j}+r_{I_{2}j}}+ \cdots +\frac{r_{I_{1}j}\left(t-\sum_{l=1}^{d}(b_{d}-b_{l})\right)}{\sum_{i \in I} r_{ij}}\right)\cdots \\
\quad \cdot \bar{F}_{Y_{I_{d}j}}\left(\frac{r_{I_{d}j}(t-\sum_{l=1}^{d}(b_{d}-b_{l}))}{\sum_{i \in I} r_{ij}}\right) & \text{ for } \sum_{l=1}^{d}(b_{d}-b_{l}) < t,
\end{cases}
\end{align*}
and by definition of NBU distributions we get
\begin{align}
\bar{F}_{T_{\mathrm{awt},j}^{I}}(t) \geq
\begin{cases}
\bar{F}_{S_{j}}\left(t \right) & \text{ for } 0 < t < b_{2}, \\
\bar{F}_{S_{j}}\left(b_{2}+\frac{r_{I_{1}j}(t-b_{2})}{r_{I_{1}j}+r_{I_{2}j}}+ \frac{r_{I_{2}j}(t-b_{2})}{r_{I_{1}j}+r_{I_{2}j}}\right) = \bar{F}_{S_{j}}( t) & \text{ for }b_{2} < t < \sum_{l=1}^{3} (b_{3}-b_{l}) , \\
\vdots \\
\bar{F}_{S_{j}}\left( b_{2}+\frac{r_{I_{1}j}2(b_{3}-b_{2})}{r_{I_{1}j}+r_{I_{2}j}}+ \cdots+\frac{r_{I_{1}j}\left(t-\sum_{l=1}^{d}(b_{d}-b_{l})\right)}{\sum_{i \in I} r_{ij}}+\cdots \right.
\\ \quad \left. + \frac{r_{I_{d}j}\left(t-\sum_{l=1}^{d}(b_{d}-b_{l})\right)}{\sum_{i \in I} r_{ij}}\right)= \bar{F}_{S_{j}}( t) & \text{ for } \sum_{l=1}^{d}(b_{d}-b_{l}) < t.
\end{cases}
\label{eq: thm2 ccdf awt}
\end{align}
It then follows that the expected aggregate weighted amount of time invested in the service of a job is larger than or equal to the mean size of a single job instance, i.e.,
\begin{align}
\sum_{i=1}^{N} r_{ij} \mathbb{E}[T_{ij}] = \mathbb{E}[X_{j}]\int_{t=0}^{\infty} \bar{F}_{T_{\mathrm{awt},j}}(t) \mathrm{d}t &\geq \mathbb{E}[X_{j}]\int_{t=0}^{\infty} \bar{F}_{S_{j}}( t) \mathrm{d}t = \mathbb{E}[X_{j}] \mathbb{E}[S_{j}],
\label{eq: thm2 expected aggregated weighted time}
\end{align}
and substituting $\tau_{ij} = \lambda p_{j} \mathbb{E}[T_{ij}]$ yields Equation~\eqref{ineqpertype}
\begin{align*}
\sum_{i=1}^{N} r_{ij} \tau_{ij} \geq \lambda p_{j} \mathbb{E}[X_{j}] \mathbb{E}[S_{j}].
\end{align*}
Summing over all the job types gives Equation~\eqref{ineqtotal}
\begin{align}
\sum_{j=1}^{M} \sum_{i=1}^{N} r_{ij} \tau_{ij} \geq \lambda \sum_{j=1}^{M} p_{j} \mathbb{E}[X_{j}] \mathbb{E}[S_{j}].
\label{eq: thm2 fraction of time}
\end{align}
Note that at this point, by Proposition~\ref{propo: prob allocation} with $\epsilon=0$, it follows that the stability region for $d=1$ is larger than or equal to the stability region for $d>1$ in the case of NBU distributed speed variations. 

In the case of a strictly NBU distribution Equation~\eqref{eq: thm2 ccdf awt} is a strict inequality if two or more servers are serving this particular job, i.e., for $t > b_{2}$. We proceed by proving that Equation~\eqref{eq: thm2 fraction of time} holds with strict inequality.\\

Note that we can write Equation~\eqref{eq: thm2 expected aggregated weighted time} as
\begin{align}
\sum_{i=1}^{N} r_{ij} \mathbb{E}[T_{ij}] &= \mathbb{E}[X_{j}] \int_{t=0}^{\infty} \bar{F}_{T_{\mathrm{awt},j}}(t) \mathrm{d}t \geq \mathbb{E}[X_{j}] \Big( \int_{t=0}^{\infty} \bar{F}_{S_{j}}(t) \mathrm{d}t + \mathbb{E}[L(d,S_{j},B)] \Big),
\label{eq: thm2 strict expected aggregated weighted time}
\end{align}
where the latter expectation is with respect to $S$ and where
\begin{align*}
L(d,S_{j},\boldsymbol{b}) = \int_{t=\sum_{l=1}^{d}(b_{d}-b_{l})}^{\infty} \Big( \bar{F}_{T_{\mathrm{awt},j}}(t)-\bar{F}_{S_{j}}( t)\Big) \mathrm{d}t
\end{align*}
denotes the difference between, starting from the time a job is in service at $d$ servers, of the aggregate weighted amount of time invested in the service of an arbitrary type-$j$ job and the job size under the distributions of $X$ and $B$, where $B=(B_{2},\dots,B_{d})$ is the random variable that denotes the weighted amount of time after which the server is available, with $0 \leq B_{2} \leq \dots \leq B_{d}$ and joint probability density function $f_{B}(b_{2},\dots,b_{d})$.
Although all jobs that are in service at two or more servers contribute to the strict inequality of Equation \eqref{eq: thm2 fraction of time}, we only consider the job that is in service at all the $d$ servers.
Moreover, by Equation~\eqref{eq: thm2 ccdf awt} we know that $\bar{F}_{T_{\mathrm{awt},j}^{I}}(t) > \bar{F}_{S_{j}}(t)$ for all $t \geq \sum_{l=1}^{d}(b_{d}-b_{l})$.
Note that by Property~\ref{prop: oldest job}, if the system is non-empty, there is always a job that is served at all the servers that it has been replicated to. 
Substituting $\tau_{ij} = \lambda p_{j} \mathbb{E}[T_{ij}]$ in Equation~\eqref{eq: thm2 strict expected aggregated weighted time} and summing over all the job types gives
\begin{align}
\sum_{j=1}^{M} \sum_{i=1}^{N} r_{ij} \tau_{ij} & \geq \sum_{j=1}^{M} \Big( \lambda p_{j} \mathbb{E}[X_{j}]\mathbb{E}[S_{j}] + \min_{k \in \mathcal{N}, l \in \mathcal{M}} r_{kl} \lambda \mathbb{E}[L(d,S_{j},B)]  \Big)\nonumber \\
& \geq \lambda \sum_{j=1}^{M} p_{j} \mathbb{E}[X_{j}]\mathbb{E}[S_{j}] \Big(1+ \tfrac{\min_{k \in \mathcal{N}, l \in \mathcal{M}}r_{kl} \mathbb{E}[L(d,S_{j},B)]}{p_{j} \mathbb{E}[X_{j}]\mathbb{E}[S_{j}]} \Big).
\label{eq: thm2 sum fraction time}
\end{align}

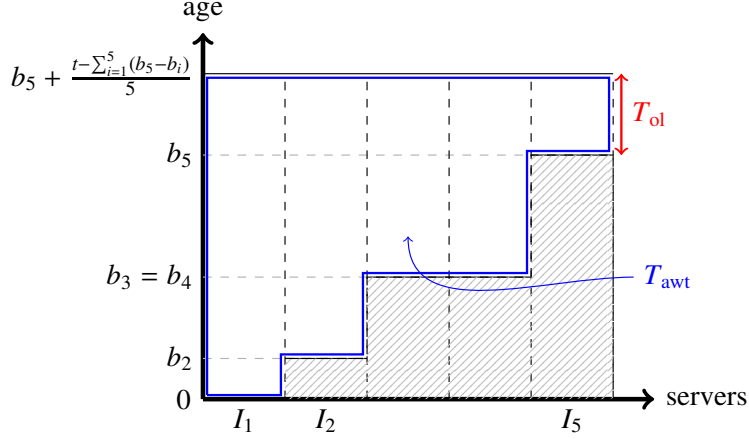
\begin{figure}[]
\centering
\resizebox{0.75\textwidth}{!}{\usetikzlibrary{patterns}
\begin{tikzpicture}[scale=1]
%\draw[help lines, color=gray!30, dashed] (0,0) grid (5.5,4.5);
\draw[->,ultra thick] (0,0) --(5.5,0) node[right]{servers};
\draw[->,ultra thick] (0,0)node[left]{$0$}--(0,4.5) node[above]{age};

\draw[] (0,4)node[left]{$b_{5}+\frac{t-\sum_{i=1}^{5}(b_{5}-b_{i})}{5}$}--(5,4);
\draw[] (4,3)--(5,3);
\draw[dashed,color=gray!60] (0,3)node[left,,color=black]{$b_{5}$} -- (4,3);
\draw[] (2,1.5)--(4,1.5);
\draw[dashed,color=gray!60] (0,1.5)node[left,,color=black]{$b_{3}=b_{4}$} -- (2,1.5);
\draw[] (1,0.5)--(2,0.5);
\draw[dashed,color=gray!60] (0,0.5)node[left,color=black]{$b_{2}$} -- (1,0.5);

\draw[dashed] (1,0)--(1,4);
\draw[dashed] (2,0)--(2,4);
\draw[dashed] (3,0)--(3,4);
\draw[dashed] (4,0)--(4,4);
\draw[dashed] (5,0)--(5,4);

\node[below] at (0.5,0) {$I_{1}$};
\node[below] at (1.5,0) {$I_{2}$};
\node[below] at (4.5,0) {$I_{5}$};

\draw[pattern=north east lines, dashed, pattern color=gray!50] (1,0.5)--(2,0.5)--(2,1.5)--(4,1.5)--(4,3)--(5,3)--(5,0)--(1,0);

%\draw[thick,color=red] (0.05,0.05)--(0.05,3.95)--(4.95,3.95)--(4.95,3.05)--(3.95,3.05)--(3.95,1.55)--(1.95,1.55)--(1.95,0.55)--(0.95,0.55)--(0.95,0.05)--(0.05,0.05);
\draw[<->,thick,color=red] (5.1,4)--node[right]{$T_{\text{ol}}$}(5.1,3);
\draw[thick,color=blue] (0.05,0.05)--(0.95,0.05)--(0.95,0.55)--(1.95,0.55)--(1.95,1.55)--(3.95,1.55) --(3.95,3.05)--(4.95,3.05)--(4.95,3.95)--(0.05,3.95)--(0.05,0.05);
\draw[<-,color=blue] (2.5,2) to[out=270, in=180] (5.25,1.5) node[right]{$T_{\text{awt}}$};
\end{tikzpicture}}
\caption{Illustration of the definition of $T_{\mathrm{awt}}$ and $T_{\mathrm{ol}}$ in case of homogeneous server speeds.}
\label{fig: T overlap}
\end{figure}

To prove the strict inequality in Equation~\eqref{eq: thm2 fraction of time} we have to show that $\mathbb{E}[ L(d,S_{j},B) ] > 0$.

Let $T_{\mathrm{ol}}(d,S_{j},\boldsymbol{b})$ denote the overlap in the service of an arbitrary type-$j$ job, see Figure \ref{fig: T overlap} for a visual interpretation, then
\begin{align}
\mathbb{E}[T_{\mathrm{ol}}(d,S_{j},B)] &= \int_{b_{2}=0}^{\infty} \dots \int_{b_{d}=0}^{\infty} T_{\mathrm{ol}}(d,S_{j},\boldsymbol{b}) f_{B}(b_{2},\dots,b_{d}) \mathrm{d}b_{2}\cdots \mathrm{d}b_{d},
\label{eq: thm2 def expected overlap}
\end{align}
where
\begin{align*}
T_{\mathrm{ol}}(d,S_{j},\boldsymbol{b}) = \int_{t=\sum_{l=1}^{d}(b_{d}-b_{l})}^{\infty} \frac{1}{d} \cdot \bar{F}_{T_{\mathrm{awt},j}}(t) \mathrm{d}t.
\end{align*}
Since $\lambda \mathbb{E}[T_{\mathrm{ol}}(d,S_{j},B)] \geq \bar{\pi}_{0}$, we can get a lower bound for the expected overlap
\begin{align*}
\mathbb{E}[T_{\mathrm{ol}}(d,S_{j},B)] \geq \frac{\bar{\pi}_{0}}{\lambda} &\geq \frac{1}{\lambda} \sum_{j=1}^{M} \frac{\lambda p_{j} \mathbb{E}[\min \{Y_{1j},\dots,Y_{dj}\}]}{\sum_{i=1}^{N} r_{ij}} \geq  \frac{\mathbb{E}[\min \{Y_{1j},\dots,Y_{dj}\}]}{\max_{j \in \mathcal{M}} \sum_{i=1}^{N} r_{ij}}.
\end{align*}
Observe that from this lower bound and Equation~\eqref{eq: thm2 def expected overlap} it follows that there exists $\tau(\delta) < \max \{\mathcal{R}_{S_{j}}\}$ such that $\mathbb{P}(B_{d} < \tau(\delta)) \geq \delta$, otherwise $\mathbb{E}[T_{\mathrm{ol}}(d,S_{j},B)]$ is too small.
Using this we can write
\begin{align*}
\mathbb{E}[L(d,S_{j},B)] &\geq \mathbb{P}(B_{d} < \tau(\delta)) \cdot \mathbb{E}[L(d,S_{j},B) | B_{d} < \tau(\delta)] \\
&\geq \delta \mathbb{E}[L(d,S_{j},B) | B_{d} < \tau(\delta)] \\
&\geq \delta \int_{t=(d-1)\tau(\delta)}^{\infty}  \Big( \bar{F}_{T_{\mathrm{awt},j}}(t) - \bar{F}_{S_{j}}(t) \Big) \mathrm{d}t  =: \delta  I(\delta) > 0.
\end{align*}
Hence,
\begin{align*}
\sum_{j=1}^{M} \sum_{i=1}^{N} r_{ij} \tau_{ij} \geq \lambda \sum_{j=1}^{M} p_{j} \mathbb{E}[X_{j}] \mathbb{E}[S_{j}] \Big(1+\frac{\min_{k \in \mathcal{N}, l \in \mathcal{M}}r_{kl} \delta I(\delta)}{p_{j}\mathbb{E}[X_{j}] \mathbb{E}[S_{j}]}\Big).
\end{align*}
Now the proof follows by Proposition~\ref{propo: prob allocation}, with $\epsilon = \min_{j \in \mathcal{M}} \frac{\min_{k \in \mathcal{N}, l \in \mathcal{M}}r_{kl} \delta I(\delta)}{p_{j}\mathbb{E}[X_{j}]\mathbb{E}[S_{j}]}$ which is bounded away from zero.
\end{proof}

\begin{remark}
We obtained a lower bound for $\mathbb{E}[L(d,S_{j},B))]$ which is strictly increasing in $d$, see Theorem~\ref{thm: maximum stability d=1 iid NBU}, but have no meaningful upper bound for this expression. Nonetheless, it would be natural to conjecture that Theorem~\ref{thm: maximum stability d=1 iid NBU} extends to the statement that the stability region is strictly decreasing in $d$.
\end{remark}

\begin{remark}
In Theorems~\ref{thm: maximum stability d=1 identical} and \ref{thm: maximum stability d=1 iid NBU} we restricted ourselves to static probabilistic assignment of the $d$ replicas. This restriction could probably be relaxed to dynamic assignments policies. Think for example of an assignment policy that replicates the job to, say, $\tilde{d}$ servers, where $\tilde{d}$ is a realization from some underlying distribution which may depend on the job type.
\end{remark}

In the next subsection we show, by providing counterexamples, that the assumptions in Theorems~\ref{thm: maximum stability d=1 identical} and \ref{thm: maximum stability d=1 iid NBU}, i.e., known job types and static probabilistic type-dependent assignment, are in fact \textit{necessary}.

\subsection{Necessary assumptions}
In this section we analyze the stability region in cases where the assumptions in Theorems \ref{thm: maximum stability d=1 identical} and \ref{thm: maximum stability d=1 iid NBU} do not all hold.

\begin{example}
\label{exam: slow fast server unknown job types}
Consider the scenario of Example~\ref{exam: slow fast servers}, i.e., $N=2$, $M=2$ and server speeds $(r_{11},r_{21})=(1,x)$ and $(r_{12},r_{22})=(x,1)$ with probabilities $p_{1}=p_{2}=0.5$, where $x < 1$. However, in this scenario the job types are unknown. In case of $d=1$, unknown job types implies that both servers are equivalent, thus the optimal static probabilistic assignment is in that case $\tilde{p}_{1j}=\tilde{p}_{2j}=0.5$ for $j=1,2$. The stability conditions, see also Equation~\eqref{eq: stability condition d=1 no job types}, are
\begin{align*}
&\lambda \sum_{j=1}^{2} p_{j} \mathbb{E}[X_{j}]\mathbb{E}[S_{j}] < 1, &\text{ for } d=2 \text{ (identical)},\\
&\lambda \sum_{j=1}^{2} p_{j} \mathbb{E}[X_{j}]\mathbb{E}\left[\min\left\{Y_{1j},\frac{Y_{2j}}{x}\right\}\right] < 1, &\text{ for } d=2 \text{ (i.i.d.)},\\
&\lambda \sum_{j=1}^{2} p_{j} \mathbb{E}[X_{j}] \mathbb{E}[S_{j}]\left(0.5 + \frac{0.5}{x}\right) < 1, &\text{ for } d=1,
\end{align*}
where $Y_{1j}$ and $Y_{2j}$ are i.i.d.\ copies of $S$. Note that $\mathbb{E}\left[\min\left\{Y_{1j},\frac{Y_{2j}}{x}\right\}\right] \leq \mathbb{E}[S_{j}]$, and that $0.5 + \frac{0.5}{x} > 1$ for $x<1$.
\end{example}

The above example shows that Theorems \ref{thm: maximum stability d=1 identical} and \ref{thm: maximum stability d=1 iid NBU} do not hold when job types cannot be distinguished.\\

\noindent \textit{Random job assignment:}\\
To achieve the largest stability region with no replication we allowed for static probabilistic assignment of jobs. 
Example~\ref{exam: slow fast server unknown job types} illustrates that Theorems \ref{thm: maximum stability d=1 identical} and \ref{thm: maximum stability d=1 iid NBU} do not hold when we restrict the assignment probabilities of jobs in case of no replication to be uniform. 

\subsection{No replication may be best for NWU speed variations}
Example~\ref{exam: slow fast servers} already showed that even for NWU speed variations, in this specific scenario, no replication gives a larger stability region than full replication. However, Example~\ref{exam: homogeneous servers} showed that in the scenario with homogeneous server speeds full replication gives a larger stability region. From both examples we conclude that in the case of known job types and NWU distributed speed variations the number of replicas that achieves the largest stability region heavily depends on the server speeds. Loosely speaking, full replication or no replication gives the largest stability region if the server speeds within a job type are balanced and unbalanced, respectively.  

\section{Full replication is best for NWU speed variations}
\label{sec: stability NWU distribution}
In this section we prove that full replication gives a larger stability region than no replication when the speed variations are NWU distributed and job types cannot be observed, see Theorem~\ref{thm: stability region d=N vs d=1}. We also discuss the possible extensions of this statement, replacing no replication by an arbitrary number of replicas, in Conjectures~\ref{conj: same starting times better} and~\ref{conj: stability region d=N vs arbitrary d}.\\ 

We first introduce some useful notation.
Consider $K = {N \choose d}$ probabilities, where each probability corresponds to assigning a job to one of the ${N \choose d}$ possible combinations of $d$ servers. Let $\boldsymbol{s}^{i} \subset \{1,\dots,N\}$ denote the set of servers corresponding to the $i$-th probability. Without loss of generality, we suppose that $\tilde{p}_{1}$ corresponds to the set of servers $\boldsymbol{s}^{1}=\{1,\dots,d\}$, $\tilde{p}_{2}$ corresponds to the set of servers $\boldsymbol{s}^{2}=\{1,\dots,d-1,d+1\}$ and finally $\tilde{p}_{K}$ corresponds to the set of servers $\boldsymbol{s}^{K}=\{N-d+1,\dots,N\}$, with $\sum_{i=1}^{K} \tilde{p}_{i} = 1$. 

For brevity, we further define $\gamma_{i} = \sum_{j = 1}^{M} p_{j} \mathbb{E}[X_{j}] \gamma_{ij}$, with
\begin{align*}
\gamma_{ij} = \sum_{h: i \in \boldsymbol{s}^{h}} \frac{\tilde{p}_{h}}{\sum_{h^{*}: i \in \boldsymbol{s}^{h^{*}}} \tilde{p}_{h^{*}}} \theta_{ijh},
\end{align*}
and
\begin{align*}
\theta_{ijh} = \mathbb{E}\left[\min \left\{\frac{Y_{1j}}{r_{s^{h}_{1}j}},\dots, \frac{Y_{dj}}{r_{s^{h}_{d}j}}\right\}\right],
\end{align*}
representing the expected execution time per unit size for a type-$j$ job assigned to the set of servers $s^h \ni i$ if all $d$~replicas were to start at the same time.  Thus, $\gamma_{i}$ may be interpreted as a \emph{proxy} for the load associated with an arbitrary job assigned to server~$i$.
In case the random speed variation $S_{j}$ is exponentially distributed, the expression for $\theta_{ijh}$ reduces to
\begin{align*}
\hat{\theta}_{ijh} = \frac{\mathbb{E}[S_{j}]}{\sum_{l=1}^{d} r_{s_{l}^{h}j}},
\end{align*}
and we will add a hat to the coefficients $\gamma_{i}$ in that case accordingly and informally refer to these as the \emph{exponential} load values.
For $d = 1$, the expression for $\gamma_{ij}$ simplifies to $\mathbb{E}[S_{j}] / r_{ij}$, yielding
\begin{align*}
\tilde{\gamma}_{i} = \sum_{j = 1}^{M} p_j \mathbb{E}[X_{j}] \frac{\mathbb{E}[S_{j}]}{r_{ij}},
\end{align*}
which is in fact the \emph{exact} load in that case.
For $d = N$, the values of $\gamma_{ij}$ are all equal to
\begin{align*}
\theta_{j} = \mathbb{E}[\min\left\{\frac{Y_{1j}}{r_{1j}}, \dots, \frac{Y_{Nj}}{r_{Nj}}\right\}],
\end{align*}
and hence the values of~$\gamma_{i}$ are all equal to
\begin{align*}
\gamma_{0} =  \sum_{j = 1}^{M} p_{j} \mathbb{E}[X_{j}] \theta_{j},
\end{align*}
which is also the \emph{exact} load since all the $N$~replicas are guaranteed to start at the same time.
Finally note that in case the random speed variation $S_{j}$ is exponentially distributed, the expression for $\theta_{j}$ simplifies to
\begin{align*}
\hat{\theta}_{j} = \frac{\mathbb{E}[S_j]}{\sum_{i = 1}^{N} r_{ij}},
\end{align*}
yielding
\begin{align*}
\hat{\gamma}_{0} = \sum_{j = 1}^{M} p_{j} \mathbb{E}[X_{j}] \frac{ \mathbb{E}[S_{j}]}{\sum_{i = 1}^{N} r_{ij}}.
\end{align*}

In the next theorem we prove that full replication gives a (strictly) larger stability region than no replication when the speed variations are (strictly) NWU distributed. 

\begin{theorem}
\label{thm: stability region d=N vs d=1}
In the case of unknown job types, the stability region for $d=N$ is larger than or equal to (respectively, strictly larger than) the stability region for $d=1$ with NWU (respectively, strictly NWU) distributed speed variations and static probabilistic assignment (which cannot depend on the job type) of the $d$ replicas. 
\end{theorem}
\begin{proof}
For $d=1$, the stability condition is $\max_{i \in \mathcal{N}} \tilde{p}_{i} \tilde{\gamma}_{i} < 1$ for some probabilities $\tilde{p}_{i}$, see~\eqref{eq: stability condition d=1 no job types}. For $d=N$, the stability condition is $\gamma_{0} < 1$, see~\eqref{eq: stability condition d=N no job types}.
For all NWU distributed speed variations, see for example~\cite[Sec.~1.6]{S-CMQST}, we have
\begin{align*}
\theta_{j} = \mathbb{E}\left[\min \left\{\frac{Y_{1j}}{r_{1j}},\dots, \frac{Y_{Nj}}{r_{Nj}}\right\}\right] \leq \frac{\mathbb{E}[S_{j}]}{\sum_{i=1}^{N} r_{ij}} = \hat{\theta}_{j},
\end{align*}
and hence $\gamma_{0} \leq \hat{\gamma}_{0}$,
which is a strict inequality in the case of a strictly NWU distribution. 
The remainder of the proof follows as a special case of Lemma~\ref{lem: underlying algebraic equation} stated below, noting that $\hat{\gamma}_{i} \sum_{h : i \in \boldsymbol{s}^{h}} \tilde{p}_{h} = \tilde{p}_{i} \tilde{\gamma}_{i}$ when $d=1$. 
\end{proof}

Lemma~\ref{lem: underlying algebraic equation} establishes a fundamental algebraic inequality for the exponential load values which will be of key importance throughout the remainder of this section as well. 

\begin{lemma}
\label{lem: underlying algebraic equation}
For all choices of the probabilities $\tilde{p}_{k}$, $k=1,\dots,K$, we have
\begin{align}
\label{eq: equation z gamma}
\hat{\gamma}_{0} \leq \max_{i \in \mathcal{N}} \hat{\gamma}_{i} \sum_{h : i \in \boldsymbol{s}^{h}} \tilde{p}_{h}.
\end{align}
\end{lemma}
\begin{proof}
If we minimize the right-hand side in Equation~\eqref{eq: equation z gamma} by setting
\begin{align*}
\hat{\gamma}_{1} \sum_{h : 1 \in \boldsymbol{s}^{h}} \tilde{p}_{h} = \dots = \hat{\gamma}_{N} \sum_{h : N \in \boldsymbol{s}^{h}} \tilde{p}_{h},
\end{align*} 
then it follows that this term is equal to  
\begin{align*}
\frac{d}{\frac{1}{\hat{\gamma}_{1}} + \dots + \frac{1}{\hat{\gamma}_{N}}} = \frac{d \prod_{k =1}^{N} \hat{\gamma}_{k}}{\sum_{l=1}^{N} \prod_{k \neq l} \hat{\gamma}_{k}} = \hat{\gamma}_{i} \sum_{h : i \in \boldsymbol{s}^{h}} \tilde{p}_{h}.
\end{align*}
Note that for $d=1$ we can get an explicit expression for the probabilities since we have a system of $N$ equations with $N$ unknowns, i.e., $\tilde{p}_{i} = \frac{\prod_{k=1, k \neq i}^{N} \hat{\gamma}_{k}}{\sum_{l=1}^{N} \prod_{k \neq l} \hat{\gamma}_{k}}$. For $d>1$ we have $K$ unknowns which makes the system of equations underdetermined. 
Thus, Equation~\eqref{eq: equation z gamma} is equivalent to
\begin{align}
\hat{\gamma}_{0} \leq \frac{d}{\frac{1}{\hat{\gamma}_{1}} + \dots + \frac{1}{\hat{\gamma}_{N}}} \Leftrightarrow \frac{d}{\hat{\gamma}_{0}} \geq  \frac{1}{\hat{\gamma}_{1}} + \dots + \frac{1}{\hat{\gamma}_{N}}.
\label{eq: expression gamma z tilde}
\end{align}

We can rewrite the right-hand side of the expression to
\begin{align*}
\frac{1}{\hat{\gamma}_{1}} + \dots + \frac{1}{\hat{\gamma}_{N}} &= \frac{1}{\sum_{h : 1 \in \boldsymbol{s}^{h}} \frac{\tilde{p}_{h}}{\sum_{h^{*}: 1 \in \boldsymbol{s}^{h^{*}}} \tilde{p}_{h^{*}}} (\frac{\hat{\gamma}_{01}}{x_{h1}} + \dots + \frac{\hat{\gamma}_{0M}}{x_{hM}})} + \dots + \frac{1}{\sum_{h : N \in \boldsymbol{s}^{h}} \frac{\tilde{p}_{h}}{\sum_{h^{*}: N \in \boldsymbol{s}^{h^{*}}} \tilde{p}_{h^{*}}} (\frac{\hat{\gamma}_{01}}{x_{h1}} + \dots + \frac{\hat{\gamma}_{0M}}{x_{hM}})},
\end{align*}
where $\hat{\gamma}_{0j}=p_{j} \mathbb{E}[X_{j}] \frac{ \mathbb{E}[S_{j}]}{\sum_{i = 1}^{N} r_{ij}}$ and $x_{lj} = \frac{\sum_{i=1}^{d}r_{s^{l}_{i}j}}{\sum_{i=1}^{N} r_{ij}}$ for $l=1,\dots,K$ and $j=1,\dots,M$. The above expression is concave in $x_{ij}$ when fixing the values $\sum_{i=1}^{N} r_{ij}$, $j=1,\dots,M$, and is therefore maximized for $x_{l1}=\dots=x_{lM}$ for all $l=1,\dots,K$, for which the expression is equal to $\frac{d}{\hat{\gamma}_{0}}$.
\end{proof}

Extending Theorem~\ref{thm: stability region d=N vs d=1} to all values of $1 \leq d \leq N$ is challenging. One of the key difficulties is that the various replicas do not necessarily start at the same time as a result of queueing, making it impossible to determine the exact load values when $d$ is strictly between $1$ and $N$. Establishing suitable lower bounds for the load values would provide a potential way to circumvent that issue. The next lemma presents a possible path in that direction by showing that the minimum expected aggregate weighted load is achieved when all replicas start at exactly the same time. 
 
\begin{lemma}
\label{lem: expected aggregate weighted amount of time}
For any number of replicas and NWU distributed job sizes the expected aggregate weighted amount of time invested in the service of a job is minimized when all the replicas start at exactly the same time. Specifically, for each job type $j$,
\begin{align*}
\sum_{i \in \boldsymbol{s}^{h}} r_{ij} \mathbb{E}[T_{ij}] \geq \sum_{i \in \boldsymbol{s}^{h}} r_{ij} \theta_{ijh}.
\end{align*}
\end{lemma}
\begin{proof}
Observe that for NWU distributions, Equation~\eqref{eq: thm2 ccdf awt} changes to
\begin{align*}
\bar{F}_{T_{\mathrm{awt},j}^{I}}(t) =
\begin{cases}
\bar{F}_{Y_{I_{1}j}}( t) \geq \bar{F}_{Y_{I_{1}j}}\left(\frac{r_{I_{1}j}t}{\sum_{i \in I} r_{ij}}\right) \cdots \bar{F}_{Y_{I_{d}j}}\left(\frac{r_{I_{d}j}t}{\sum_{i \in I} r_{ij}}\right) & \text{ for } 0 < t < b_{2}, \\
\bar{F}_{Y_{I_{1}j}}\left( b_{2}+\frac{r_{I_{1}j}(t-b_{2})}{r_{I_{1}j}+r_{I_{2}j}}\right) \cdot \bar{F}_{Y_{I_{2}j}}\left(\frac{r_{I_{2}j}(t-b_{2})}{r_{I_{1}j}+r_{I_{2}j}}\right) \\
\quad \geq \bar{F}_{Y_{I_{1}j}}\left(\frac{r_{I_{1}j}t}{\sum_{i \in I} r_{ij}}\right) \cdots \bar{F}_{Y_{I_{d}j}}\left(\frac{r_{I_{d}j}t}{\sum_{i \in I} r_{ij}}\right) & \text{ for }b_{2} < t < \sum_{l=1}^{2} (b_{3}-b_{l}), \\
\vdots \\
\bar{F}_{Y_{I_{1}j}}\left(b_{2}+\frac{r_{I_{1}j}2(b_{3}-b_{2})}{r_{I_{1}j}+r_{I_{2}j}}+ \cdots +\frac{r_{I_{1}j}\left(t-\sum_{l=1}^{d}(b_{d}-b_{l})\right)}{\sum_{i \in I} r_{ij}}\right) \\
\quad \cdots \bar{F}_{Y_{I_{d}j}}\left(\frac{r_{I_{d}j}(t-\sum_{l=1}^{d}(b_{d}-b_{l}))}{\sum_{i \in I} r_{ij}}\right) \\
\quad \geq \bar{F}_{Y_{I_{1}j}}\left(\frac{r_{I_{1}j}t}{\sum_{i \in I} r_{ij}}\right) \cdots \bar{F}_{Y_{I_{d}j}}\left(\frac{r_{I_{d}j}t}{\sum_{i \in I} r_{ij}}\right) & \text{ for } \sum_{l=1}^{d}(b_{d}-b_{l}) < t.
\end{cases}
\end{align*}
By definition of NWU distributions, 
\begin{align}
\sum_{i \in I} r_{ij} \mathbb{E}[T_{ij}] = \int_{t=0}^{\infty} \bar{F}_{T_{\mathrm{awt},j}^{I}}(t) \mathrm{d}t &\geq \int_{t=0}^{\infty} \bar{F}_{Y_{I_{1}j}}\left(\frac{r_{I_{1}j}t}{\sum_{i \in I} r_{ij}}\right) \cdots \bar{F}_{Y_{I_{d}j}}\left(\frac{r_{I_{d}j}t}{\sum_{i \in I} r_{ij}}\right) \mathrm{d}t \nonumber \\
&= \sum_{i \in I} r_{ij} \mathbb{E}\left[\min \left\{\frac{Y_{I_{1}j}}{r_{I_{1}j}},\dots, \frac{Y_{I_{d}j}}{r_{I_{d}j}}\right\}\right] = \sum_{i \in I} r_{ij} \theta_{ijh},
\label{eq: aggregate weight inequalities}
\end{align}
with $\boldsymbol{s}^{h} = I$.
This implies that the minimum expected aggregate weighted amount of time invested in the service of a job is achieved when all replicas start at exactly the same time. 
\end{proof}

Now observe that if $\mathbb{E}[T_{ij}] \geq \theta_{ijh}$ for all $i \in I = \boldsymbol{s}^{h}$, then similar arguments as in Theorem~\ref{thm: stability region d=N vs d=1} and Lemma~\ref{lem: underlying algebraic equation} would yield that the stability region for $d=N$ is larger than for any $1<d<N$ as well. Unfortunately, these detailed inequalities cannot be deduced from the aggregate weighted inequalities in \eqref{eq: aggregate weight inequalities} without further conditions. 
This leads to the next conjecture, which is also illustrated in Figure~\ref{fig: stability simulation and same starting times FS} for Weibull$(\lambda_{\mathrm{w}}=1.128,k=2)$ (NBU), exponential and Weibull$(\lambda_{\mathrm{w}}=0.5,k=0.5)$ (NWU) distributed speed variations.

\begin{figure}
  \centering
\resizebox{0.75\textwidth}{!}{\newcommand{\dataFigure}{TikZFigures/ExpectedLatency_vs_SameStartingTimes_FS.csv}

\begin{tikzpicture}
		\begin{axis}[
			xlabel=$\lambda$,
			ylabel=Expected latency,
			ymin=0,
			ymax=20,
			no markers,
			legend pos= outer north east,
			legend cell align=left]
		\addlegendimage{empty legend}
		\addplot+ table [x=lambda,y=WeibullNBU, col sep=comma]{\dataFigure};
		\addplot+ table [x=lambda,y=Exp, col sep=comma]{\dataFigure};
		\addplot+ table [x=lambda,y=WeibullNWU, col sep=comma]{\dataFigure};
		
		\pgfplotsset{cycle list shift=-3}
		\addplot+[dashed] coordinates {(1.405,0) (1.405,20)};
		\addplot+[dashed] coordinates {(1.929,0) (1.929,20)};
		\addplot+[dashed] coordinates {(3.793,0) (3.793,20)};
		\addlegendentry{Distribution:}
		\addlegendentry{Weibull (NBU)}
		\addlegendentry{Exponential}
		\addlegendentry{Weibull (NWU)}
		\end{axis}
\end{tikzpicture}}
\caption{Expected latency for the scenario of Example~\ref{exam: slow fast server unknown job types} with $N=3$ servers, $d=2$ replicas and $r_{\mathrm{slow}}=0.5$ for various distributions for the speed variation with $\mathbb{E}[S_{j}]=1$ and $\mathbb{E}[X_{j}]=1$ for $j=1,2$. Assignment $\tilde{p}_{k}=\frac{1}{3}$, for $k=1,2,3$, which is optimal. The dashed lines represent the stability condition for the load at server $i$ equal to $\gamma_{i}$, for $i=1,2,3$.}
\label{fig: stability simulation and same starting times FS}
\end{figure}
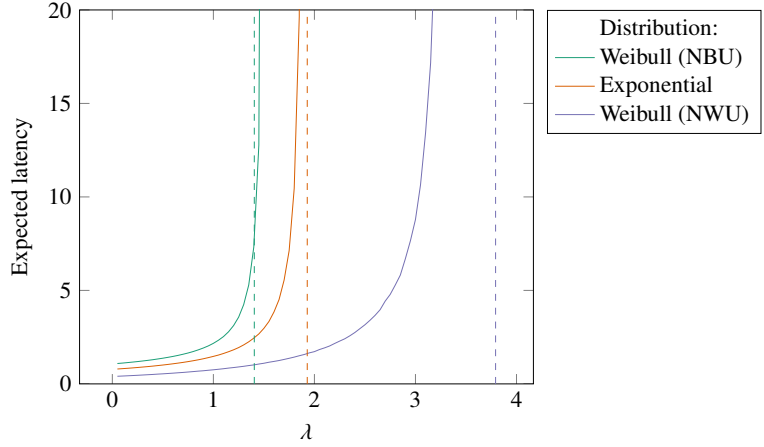

\begin{comment}
\begin{figure}
  \centering
\resizebox{0.75\textwidth}{!}{\input{TikZFigures/ExpectedLatency_vs_SameStartingTimes_H.tex}}
\caption{Expected latency for the scenario of $N=3$ heterogeneous servers, $d=2$ replicas and $r=(1,0.5,0.25)$ for various distributions for the speed variation with $\mathbb{E}[S]=1$ and $\mathbb{E}[X]=1$. Assignment $\tilde{p}_{k}=\frac{1}{3}$, for $k=1,2,3$, which is not optimal. The dashed lines represent the stability condition for the load at server $i$ equal to $\gamma_{i}$, for $i=1,2,3$.}
\label{fig: stability simulation and same starting times H}
\end{figure}
\end{comment}

\begin{conjecture}
\label{conj: same starting times better}
For (strictly) NWU distributed speed variations and unknown job types, the load at server $i$ of the system with $1 < d < N$ replicas, denoted by $\bar{\gamma}_{i}$, is bounded by $\bar{\gamma}_{i} \geq (>) \gamma_{i}$, for all $i=1,\dots,N$. 
\end{conjecture}

In Figure~\ref{fig: stability simulation and same starting times FS} it can be seen that Conjecture~\ref{conj: same starting times better} cannot be extended to NBU distributed speed variations, i.e., for the NBU Weibull distribution the stability condition of the original system (the expected latency is depicted with a solid lime green line) seems tighter than the stability condition in the system where all the $d$ replicas were to start at the same time (dashed lime green line). 

\begin{conjecture}
\label{conj: stability region d=N vs arbitrary d}
In the case of unknown job types, the stability region for $d=N$ when the job types are unknown is larger than or equal to (respectively, strictly larger than) the stability region for $1 \leq d < N$ with identical servers and NWU (respectively, strictly NWU) distributed speed variations where replicas are assigned to $d$ servers selected uniformly at random (without replacement). 
\end{conjecture}

Conjecture~\ref{conj: stability region d=N vs arbitrary d} is supported by the observation that Conjecture~\ref{conj: same starting times better} implies 
\begin{align*}
\max_{i \in \mathcal{N}} \bar{\gamma}_{i} \sum_{h : i \in \boldsymbol{s}^{h}} \tilde{p}_{h} \geq \max_{i \in \mathcal{N}} \gamma_{i} \sum_{h : i \in \boldsymbol{s}^{h}} \tilde{p}_{h},
\end{align*}
while Lemma~\ref{lem: underlying algebraic equation} gives $\hat{\gamma}_{0} \leq \max_{i \in \mathcal{N}} \hat{\gamma}_{i} \sum_{h : i \in \boldsymbol{s}^{h}} \tilde{p}_{h}$. If Conjecture~\ref{conj: same starting times better} is true, it would thus suffice to establish the equivalence relation 
\begin{align}
&\frac{1}{\gamma_{1}} + \dots + \frac{1}{\gamma_{N}} \leq \frac{d}{\gamma_{0}} \Leftrightarrow  \frac{1}{\hat{\gamma}_{1}} + \dots + \frac{1}{\hat{\gamma}_{N}} \leq \frac{d}{\hat{\gamma}_{0}}.
\label{eq: equivalence relation conjecture}
\end{align}

For identical servers with uniform selection of the servers we have that
\begin{align*}
\gamma &= \gamma_{1} = \dots = \gamma_{N} = \sum_{j=1}^{M} p_{j} \mathbb{E}[X_{j}] \sum_{h: i \in \boldsymbol{s}^{h}} \frac{\theta_{ijh}}{d}, \\
\hat{\gamma} &= \hat{\gamma}_{1} = \dots = \hat{\gamma}_{N} = \sum_{j=1}^{M} p_{j} \mathbb{E}[X_{j}] \sum_{h: i \in \boldsymbol{s}^{h}} \frac{\hat{\theta}_{ijh}}{d}.
\end{align*}
Substituting these in Equation~\eqref{eq: equivalence relation conjecture} gives
\begin{align*}
\frac{N}{\gamma} \leq \frac{d}{\gamma_{0}} \Leftrightarrow \frac{N}{\hat{\gamma}} \leq \frac{d}{\hat{\gamma}_{0}},
\end{align*}
or equivalently
\begin{align*}
\frac{\gamma}{N} \geq \frac{\gamma_{0}}{d} \Leftrightarrow \frac{\hat{\gamma}}{N} \geq \frac{\hat{\gamma}_{0}}{d}.
\end{align*}
If we look at the difference, we get
\begin{align}
\left(\frac{\hat{\gamma}}{N} - \frac{\gamma}{N} \right) - \left( \frac{\hat{\gamma}_{0}}{d} - \frac{\gamma_{0}}{d} \right).
\label{eq: difference term conjecture}
\end{align}
Now observe that $\frac{\hat{\gamma}}{N} = \frac{\hat{\gamma}_{0}}{d}$ and therefore the term in Equation~\eqref{eq: difference term conjecture} simplifies to
\begin{align*}
\frac{\gamma_{0}}{d} - \frac{\gamma}{N}.
\end{align*}
The last expression is negative since for NWU distributions, see for example~\cite[Sec.~1.6]{S-CMQST}, we have
\begin{align*}
N \theta_{j} \leq d \sum_{h: i \in \boldsymbol{s}^{h}} \frac{\theta_{ijh}}{d}, 
\end{align*}
for all $j = 1,2,\dots,M$.\\

\begin{comment}
Extending Conjecture~\ref{conj: stability region d=N vs arbitrary d} to all scenarios for the server speeds seems hard. Observe that the crucial part in our proof techniques is proving the equivalence with Equation~\eqref{eq: equation z gamma}, which in the case of general server speeds should be done without an explicit expression for the probabilities $\tilde{p}_{h}$.\\
\end{comment}

In the next subsection we will show that even for NBU distributed speed variations full replication may give the largest stability region when job types cannot be observed. This demonstrates that unpredictability in speeds induced by uncertainty in job types can create a strong rationale for replication, even when the random speed variations do not. 
More specifically, we give examples illustrating that the number of replicas that yields the largest stability region depends on the server speeds. 

\subsection{Full replication may be best for NBU speed variations}
In Section~\ref{sec: stability NBU distribution} we proved that the stability region is largest for $d=1$ when the speed variations are NBU and job types can be distinguished. We now show that the complete opposite may be true when job types cannot be observed. More specifically, we will prove that even with NBU random speed variations in some scenarios full replication gives the largest stability region when the uncertainty in the systematic speed variations is sufficiently significant in some suitable sense.

Consider the scenario where job type $j$, for $j=1,\dots,N$, is fast on server $j$, i.e., server speed $r_{\mathrm{fast}}$, and slow on the other servers, i.e., server speed $r_{\mathrm{slow}}$ (see Example~\ref{exam: slow fast server unknown job types} with $N=2$ servers, $r_{\mathrm{fast}}=1$ and $r_{\mathrm{slow}}=x$). We refer to this scenario as the \textit{FS} (Fast-Slow) scenario. 

\begin{theorem}
\label{thm: maximum stability d=N iid NBU extreme slow}
In the case of unknown job types, the stability region for $d=N$ is larger than the stability region for $d<N$ in the FS scenario with NBU distributed speed variations and static probabilistic assignment (which cannot depend on the job type) of the $d$ replicas, when the ratio $\frac{r_{\mathrm{slow}}}{r_{\mathrm{fast}}}=x \downarrow 0$.
\end{theorem}
\begin{proof}
Note that the stability region for $d=N$, given by Equation~\eqref{eq: stability condition d=N no job types}, does not depend on the value of $x$.
Now, for $d<N$, the probability of assigning all replicas of a type-$j$ job to slow servers is strictly larger than $0$. For the expected service requirement of this job, denoted by $\mathbb{E}[B^{*}_{j}]$, it follows that $\mathbb{E}[B^{*}_{j}] \geq \frac{\mathbb{E}[X_{j}]\mathbb{E}[\min\{Y_{1j},\dots,Y_{dj}\}]}{x}$. 
\end{proof}

\begin{figure}
  \centering
\resizebox{\textwidth}{!}{\newcommand{\dataFigureA}{TikZFigures/ExpectedLatency_Weibull_k_2_rslow_01.csv}
\newcommand{\dataFigureB}{TikZFigures/ExpectedLatency_Weibull_k_2_rslow_05.csv}

\begin{tabular}{@{}cc@{}}
\begin{tikzpicture}
		\begin{axis}[
			xlabel=$\lambda$,
			ylabel=Expected latency,
			ymin=0,
			ymax=20,
			no markers]
		\addplot table [x=lambda,y=d1, col sep=comma]{\dataFigureA};
		\addplot table [x=lambda,y=d2, col sep=comma]{\dataFigureA};
		\addplot table [x=lambda,y=d3, col sep=comma]{\dataFigureA};
		\end{axis}
\end{tikzpicture}
&
\begin{tikzpicture}
		\begin{axis}[
			xlabel=$\lambda$,
			ylabel=Expected latency,
			ymin=0,
			ymax=20,
			no markers,
			legend pos= outer north east,
			legend cell align=left]
		\addlegendimage{empty legend}
		\addplot table [x=lambda,y=d1, col sep=comma]{\dataFigureB};
		\addplot table [x=lambda,y=d2, col sep=comma]{\dataFigureB};
		\addplot table [x=lambda,y=d3, col sep=comma]{\dataFigureB};
		\addlegendentry{Replicas:}
		\addlegendentry{$d=1$}
		\addlegendentry{$d=2$}
		\addlegendentry{$d=N=3$}
		\end{axis}
\end{tikzpicture}
\end{tabular}}
%\resizebox{0.49\textwidth}{!}{\input{TikZFigures/ExpectedLatency_Weibull_k=2_rslow=05.tex}}
\caption{Expected latency for the scenario of Example~\ref{exam: slow fast server unknown job types} with $N=3$ servers and $r_{\mathrm{rslow}}=0.1$ (left) and $r_{\mathrm{rslow}}=0.5$ (right), where the speed variations are Weibull$(\lambda_{\mathrm{w}}=1.128,k=2)$ (NBU) distributed and $\mathbb{E}[X]=1$.}
\label{fig: weibull k=2}
\end{figure}
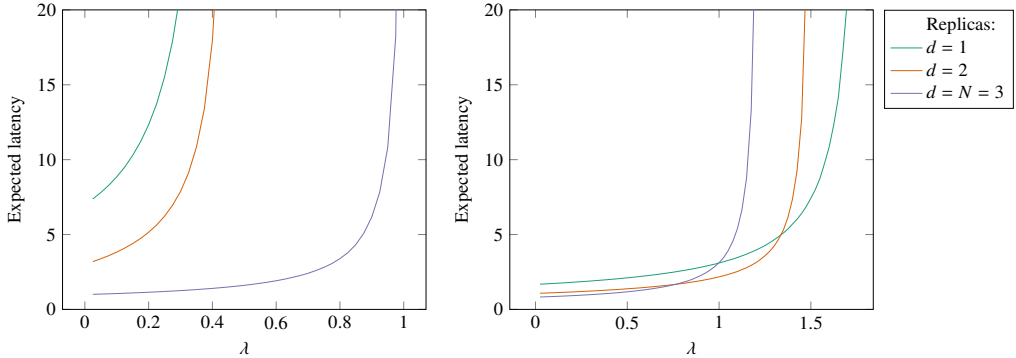

In Figure~\ref{fig: weibull k=2} (right) it can be seen that for $r_{\mathrm{slow}}$ sufficiently large $d=1$ gives the largest stability region in the special case of Weibull$(\lambda_{\mathrm{w}}=1.128,k=2)$ distributed speed variations, which belongs to the class of NBU distributions. As stated in Theorem~\ref{thm: maximum stability d=N iid NBU extreme slow} for $r_{\mathrm{slow}}$ sufficiently small we observe that $d=N=3$ gives the largest stability region. 

In \cite{AAJV-IPHDR} a similar result for the processor-sharing discipline is proved. For this discipline it is shown that redundancy can improve the stability of the system with identical replicas if the servers are sufficiently heterogeneous when the assignment probabilities are restricted to be uniform. 

\section{Conclusion and suggestions for further research}
\label{sec: conclusion}
We have proven that for c.o.c.\ redundancy scheduling with identical replicas, general job size distributions and suitable type-dependent assignment probabilities the stability region for $d=1$ is strictly larger than the stability region for $d>1$. Moreover, we established that the same statement holds in case of i.i.d.\ replicas and NBU distributed speed variations. For both identical and i.i.d.\ replicas a critical assumption is that the job types can be observed. In case of non-observable job types the stability region for $d=N$ is larger than or equal to the stability region for $d=1$ when the speed variations are NWU distributed. Under the conjecture that the stability region increases in the latter case when all replicas start at the same time, we extended the above-mentioned statement, i.e., we showed that for identical servers the stability region for $d=N$ is larger than or equal to the stability region for all $d < N$.   

In case the type identities of jobs are unknown, it may be possible to learn them, and for further research we intend to analyze the stability region when we are able to learn the job types; cf. \cite{BM-LHSS} where a learning framework is proposed to answer these questions for a different model. Ultimately, we hope to quantify the performance loss in terms of the stability region when the job types are unknown beforehand and explore how decreasing the uncertainty about the job types can increase the stability region.

\section*{Acknowledgments}
\label{sec: acknowlegdements}
The work in this paper is supported by the Netherlands Organisation for Scientific Research (NWO) through Gravitation grant NETWORKS 024.002.003.
The authors gratefully acknowledge several helpful discussions with Onno Boxma.

\bibliographystyle{plain}
%\bibliography{references}

%\begin{comment}

%\end{comment}

\begin{comment}
\section{Numerical results}
\label{app sec: numerical results}
\begin{figure}[H]
  \centering
\resizebox{0.5\textwidth}{!}{\input{TikZFigures/ExpectedLatency_Weibull_k=05_rslow=05.tex}}
\caption{Expected latency for the scenario of Example~\ref{exam: slow fast server unknown job types} with $N=3$ servers and $r_{\mathrm{rslow}}=0.5$, where the speed variations are Weibull$(\lambda=0.5,k=0.5)$ (NWU) distributed with $\mathbb{E}[X]=1$.}
\label{fig: weibull k=05 rslow=05}
\end{figure}

In Figure~\ref{fig: weibull k=05 rslow=05} it can be seen that $d=N=3$ indeed gives the largest stability region in the special case of Weibull distributed speed variations, which belongs to the class of NWU distributions. 
\end{comment}

\end{document}